\documentclass[11pt]{amsart}
\usepackage[english]{babel}
\usepackage{amssymb}
\usepackage{amsmath}
\usepackage{lscape}

\newcommand{\ignore}[1]{}

\newtheorem{dummy}{Dummy}

\newtheorem{lemma}[dummy]{Lemma}
\newtheorem{theorem}[dummy]{Theorem}
\newtheorem{proposition}[dummy]{Proposition}
\newtheorem{corollary}[dummy]{Corollary}

\theoremstyle{definition}

\newtheorem{example}[dummy]{Example}

\newtheorem{remark}[dummy]{Remark}

\author{C. Brown}
\author{S. Pumpl\"un}

\email{Christian.Brown@nottingham.ac.uk; susanne.pumpluen@nottingham.ac.uk}
\address{School of Mathematical Sciences\\
University of Nottingham\\ University Park\\ Nottingham NG7 2RD\\
United Kingdom }



\begin{document}

\title[Solvable crossed product algebras revisited]
{Solvable crossed product algebras revisited}

\begin{abstract}
For any central simple algebra over a field $F$ which contains a maximal subfield $M$
with non-trivial automorphism group $G = {\rm Aut}_F(M)$,  $G$ is solvable if and only if
the algebra contains a finite chain of subalgebras
which are generalized cyclic algebras over their centers  (field
extensions of $F$) satisfying certain conditions.  These subalgebras are related to a normal subseries of $G$.
A crossed product algebra $F$ is hence solvable if and
only if it can be constructed out of such a finite chain of subalgebras.
This result was stated for division crossed product algebras
by Petit, and overlaps with a similar result by Albert
 which, however, is not explicitly stated in these terms.
 In particular, every solvable crossed product division algebra  is a generalized cyclic
algebra over $F$.
\end{abstract}

\maketitle

%*******************************************************************************************%
%
\section*{Introduction}
%
%*******************************************************************************************%

Let $F$ be a field. A central simple algebra $A$ over $F$ of degree
$n$ is a \emph{crossed product algebra} if it contains a maximal
subfield $M$ (i.e. with $[M:F]=n$) that is Galois. To be more precise, $A$ is also
called a \emph{$G$-crossed product algebra}, if $G={\rm Gal}(M/F)$ is the Galois
group of $M/F$. Crossed product algebras play an important role in
the theory of central simple algebras: every element in the Brauer
group of $F$ is similar to a crossed product algebra, moreover, their
multiplicative structure can be described by a group action. It is
well known that any central simple algebra of degree 2, 3, 4, 6 or 12
 is a crossed product algebra. Moreover, any central
simple algebra over a local or global field is a crossed product
algebra (in that case the algebras even contain  a maximal
subfield that is cyclic).

Skew polynomial rings
have been sucessfully used in
the past to construct  central simple algebras. These
appear for instance as quotient algebras $D[t;\sigma]/(f)$ when
factoring out a two-sided ideal generated by a twisted polynomial
$f\in D[t;\sigma]$ with $D$ a finite-dimensional central division algebra over $F$ in \cite{Am2} or \cite[Sections 1.5, 1.8,
1.9]{J96}. Following Jacobson \cite[p.~19]{J96}, when $\sigma|_{F}$
has finite order $m$ and $f(t)=t^m-d\in D[t;\sigma]$, $d\in {\rm
Fix}(\sigma)^\times$, is an invariant polynomial, such a quotient
algebra is also called a \emph{generalized cyclic algebra}, and
denoted $(D,\sigma,d)$. In characteristic zero, generalized cyclic division algebras can be considered
to be the noncommutative analogue of simple algebraic field
extensions. To our knowledge, generalized cyclic division algebras appear for the
first time in a paper by Amitsur \cite{Am2}, where they are indeed called noncommutative
cyclic fields. They are examples of crossed products of central simple algebras which were introduced
 by Teichm\"uller \cite{Tei} in 1940.

 In this paper, we will revisit a result on the structure of crossed product algebras
 with solvable Galois group due to both Albert
 \cite[p.~182-187]{albert1939structure}
  and Petit \cite[Section 7]{P66}.

To be more precise, we write up the proof for Albert's result  following the approach given by
 Petit, i.e. using
generalized cyclic algebras (none of Petit's results  are proved).
 In the process, we generalize some results to central simple algebras which need neither be crossed products nor
  division algebras. In order to do so, we  extend the classical
 definition of a generalized cyclic algebra
$(D,\sigma,d)$ as we do not assume that $D$ needs to be a division
algebra.

As a special case we obtain that a
$G$-crossed product algebra is solvable if and only if it can be
constructed as a finite chain of subalgebras over $F$ which are
generalized cyclic algebras over their centers, which  are
field extensions of $F$. The  generalized cyclic algebras appearing
in this chain correspond to the normal subgroups in a chain of
normal subgroups of the solvable group $G$. We highlight how the
structure of the solvable group (i.e., its chain of normal subgroups
$G_i$) is connected to the structure of the algebra, and how each
subalgebra is related to a normal subgroup $G_i$ in the chain and the
order of the factor groups  $G_{i+1}/G_i$.

The paper is structured as follows. After the basic terminology in
Section  \ref{sec:prel} we look at the existence of
crossed product algebras and in particular, of  cyclic algebras, inside central simple algebras in Section
\ref{sec:aux}. As a byproduct, we show that even if a central division algebra $A$ over $F$ is a noncrossed product,
if it contains a maximal field extension $M$  with a non-trivial $\sigma\in G = {\rm Aut}_F(M)$ of order $h$,
then it contains a cyclic division algebra of degree $h$, and a crossed product algebra $(M,G,\mathfrak{a})$ of degree
$|G|$ as well, both of them not necessarily with center $F$, however (Theorem \ref{cor:importantI}).

 The first  results on the structure of central simple algebras which
contain a maximal subfield with  non-trivial  solvable group $G={\rm Aut}_F(M)$ are stated in Section \ref{sec:main}
(Theorems \ref{thm:Petit (29)} and \ref{thm:Petit (29) Generalised}). These algebras have certain chains of
 generalized cyclic algebras (with centers larger than $F$) as subalgebras.

As a consequence, we can show in Section \ref{sec:crossedproduct} that all solvable crossed product algebras can  be constructed as chains of
such generalized cyclic algebras and that if a central simple algebra contains a maximal subfield with $G={\rm Aut}_F(M)$
  that this $G$ is  solvable exactly if there is such a chain
 (Theorems \ref{thm:main1} and Corollary \ref{thm:Petit (29) Galois version}). In particular, every solvable $G$-crossed product division algebra is  a
generalized cyclic algebra (Corollary \ref{cor:Petit (29) Galois version}).
Some straightforward applications to admissible groups are given in Section
\ref{sec:app}. In Section \ref{sec:last} we generalize a result on crossed product algebras with Galois group
 $G\cong\mathbb{Z}_2\times\mathbb{Z}_2 $ by Albert \cite[p. 186]{albert1939structure}, cf. also \cite[Theorem 2.9.55]{J96},
to  crossed product algebras with $G$ any abelian group, and give a recipe how to construct
central division algebras containing a given Galois field extension with abelian
Galois group from a chain of generalized cyclic algebras, complementing the construction of such algebras via generic
algebras by Amitsur and Saltman described in \cite[4.6]{J96}.

Most of the results presented here are part of the first author's PhD thesis \cite{CB} written
under the supervision of the second author.

%%%%%%%%%%%%%%%%%%%%%%%%%%%%%%%%%%%%%%%%%%%%%%%%%%%%%%%%%%%%%%%%%%%%%%%%%%%%%%%%%%%%%%%%%
%
%Preliminaries
%
%%%%%%%%%%%%%%%%%%%%%%%%%%%%%%%%%%%%%%%%%%%%%%%%%%%%%%%%%%%%%%%%%%%%%%%%%%%%%%%%%%%%%%%%%%

\section{Preliminaries} \label{sec:prel}

\subsection{Twisted polynomial rings and (nonassociative) algebras}

In the following, we recall some results from \cite{J96} and
\cite{P66} for the convenience of the reader.

Let $S$ be a unital  (associative, not necessarily commutative) ring
and $\sigma$ an injective ring endomorphism of $S$.
 The \emph{twisted polynomial ring} $R=S[t;\sigma]$
is the set of twisted polynomials
$$a_0+a_1t+\dots +a_nt^n$$
with $a_i\in S$, where addition is defined term-wise and multiplication by
$$ta=\sigma(a)t \quad (a\in S).$$
For $f=a_0+a_1t+\dots +a_nt^n$ with
$a_n\not=0$ define ${\rm deg}(f)=n$ and ${\rm deg}(0)=-\infty$.
 An element $f\in R$ is \emph{irreducible} in $R$ if it is not a unit and  it has no proper factors, i.e if there do not exist $g,h\in R$ with
 ${\rm deg}(g),{\rm deg} (h)<{\rm deg}(f)$ such
 that $f=gh$. An element $f\in R$ is called \emph{invariant} (or \emph{two-sided}) if $Rf$ is a two-sided ideal
 in $R$.

 We now briefly explain how classical quotient algebras $R/Rf$, $f$ invariant, fit into the nonassociative
 setting of Petit's paper \cite{P66}:

 In the following,  we always assume that $f(t)\in S[t;\sigma]$  is
monic of degree $m>1$. Then
for all $g,f\in R$, $g\not=0$, there exist unique $r,q\in R$ such that ${\rm deg}(r)<{\rm deg}(f)$ and
$$g=qf+r,$$
 e.g. see \cite{Pu16}.

 In \cite{P66} and \cite{Pu16}, it is shown
that the additive group $S_f=\{g\in R\,|\, {\rm deg}(g)<m\}$ of
twisted polynomials of degree less that $m$ is a nonassociative unital ring together with the multiplication given by
$$g\circ h=gh \,\,{\rm mod}_r f,$$
where ${\rm mod}_r f$ denotes the
remainder ${\rm mod}_r f$ of right division by $f$. This algebra is also denoted by $R/Rf$.

Note that since the remainders are uniquely determined, the elements
in the set $S_f$ also canonically represent the elements of the left
$S[t;\sigma]$-module $S[t;\sigma]/ S[t;\sigma]f$.

 $S_0=\{a\in S\,|\, ah=ha\text{ for all } h\in S_f\}$ is a commutative subring of $S$, and $S_f$ is a unital
algebra over $S_0$.
 If $S$ is a division ring,  the structure of $S_f$ is extensively investigated in  \cite{P66},
 else see \cite{Pu16}. For instance,
 if $S$ is a division ring and the $S_0$-algebra $S_f$ is finite-dimensional, then
 $S_f$ is a division algebra if and only if $f(t)$ is irreducible \cite[(9)]{P66}.
In the following, we will only be interested in  the case that $S_f$
is a unital associative algebra,  which happens if and only if $f$ is
an \emph{invariant} polynomial in $R$, i.e. generates a two-sided
ideal $Rf$ in $R$ \cite{P66}. In that case, $S_f=R/Rf$ is the well known
quotient algebra obtained by factoring out the two-sided ideal in $R$
generated by $f$.

We will moreover only need the case that $S$ is a finite-dimensional
algebra over a field $F$ with center $F$ and only consider
automorphisms $\sigma$ of $S$ such that $\sigma|_F$ has finite order
$m$. Then, by the Theorem of Skolem-Noether, $\sigma^m$ is an inner
automorphism $I_u(y)= u y u^{-1}$ of $S$ \cite[Sec. 1.4]{J96}.

\subsection{Generalized cyclic algebras and generalized cyclic extensions}\label{sec:gca}

Let $S$ be a finite-dimensional simple algebra of degree $n$ over
its center $F=C(S)$,  and $\sigma\in {\rm Aut}(S)$ such that
$\sigma|_{F}$ has finite order $m$ and fixed field $F_0={\rm
Fix}(\sigma)$.

Generalizing Jacobson's definition \cite[p.~19]{J96}, which assumes
that $S$ is a division algebra,
 we define a \emph{generalized cyclic algebra} as an associative algebra
of the type $S_f=S[t;\sigma]/S[t;\sigma]f(t)$  which is constructed
using an invariant twisted polynomial
$$f(t)=t^m-d\in S[t;\sigma],$$
with $d\in {\rm Fix}(\sigma)^\times$ non-zero.

We write $S_f=(S,\sigma, d)$ for this algebra. $(S,\sigma, d)$ is a
central simple algebra over $ F_0={\rm Fix}(\sigma)$ of degree $mn$
and the  centralizer of $S$ in
$(S,\sigma, d)$ is $F$ (\cite[p.~20]{J96} if $S$ is division, else \cite{T}).

Note that this definition canonically generalizes the one of a cyclic
algebra $(F/F_0,\sigma,d)$, where $ f(t)=t^m-d\in F[t;\sigma]$, and
$F/F_0$ is a cyclic Galois extension of degree $m$ with Galois group
$G=<\sigma>$. This is the algebra
$S_f=F[t;\sigma]/F[t;\sigma](t^m-d) $, cf. \cite[p.~19]{J96} or \cite[p.~13-13]{P66}.
This case appears when $S=F$ above.

Generalized cyclic algebras are a special case of generalized crossed products, i.e.
crossed products of simple algebras cf. for instance  \cite[p.~35]{H},
 \cite{KY}, \cite{T}. We will mostly need crossed products involving Galois fields:

\subsection{Crossed product algebras}
Let $F$ be a field and  $A$ be a (finite-dimensional) central simple
algebra over $F$ of degree $n$. $A$ is called a \emph{$G$-crossed
product algebra} or \emph{crossed product algebra}
 if it contains a maximal field extension $K/F$  which is Galois with Galois group $G={\rm Gal}(K/F)$.

 Equivalently, we can define a ($G$-)crossed product algebra $(M,G,\mathfrak{a})$ over $F$ via factor sets starting with a finite Galois field extension
 as follows: Take a finite Galois field extension  $M/F$ of degree $n$ with
Galois group $G$. Suppose $\{ a_{\sigma,\tau} \ \vert \ \sigma, \tau
\in G \}$ is a set of elements of $M^{\times}$ such that
\begin{equation} \label{eqn:Crossed product criteria 1}
a_{\sigma,\tau} a_{\sigma \tau, \rho} = a_{\sigma,\tau \rho}
\sigma(a_{\tau,\rho}),
\end{equation}
for all $\sigma, \tau, \rho \in G$. Then a map $\mathfrak{a}: G
\times G \rightarrow M^{\times}, \ (\sigma,\tau) \mapsto
a_{\sigma,\tau}$, is called a \emph{factor set} or \emph{2-cocycle}
of $G$.

An associative multiplication is defined on the $F$-vector space
$\bigoplus_{\sigma \in G} M x_{\sigma}$ by
\begin{equation} \label{eqn:Crossed product criteria 2}
x_{\sigma} m = \sigma(m) x_{\sigma},
\end{equation}
\begin{equation} \label{eqn:Crossed product criteria 3}
x_{\sigma} x_{\tau} = a_{\sigma, \tau} x_{\sigma \tau},
\end{equation}
for all $m \in M$, $\sigma, \tau \in G$. This way $\bigoplus_{\sigma
\in G} M x_{\sigma}$ becomes an associative central simple
$F$-algebra that contains a maximal subfield isomorphic to $M$. This
algebra is denoted by $(M,G,\mathfrak{a})$ and is a $G$-crossed
product algebra over $F$.
 If $G$ is solvable then $A$ is also called a  \emph{solvable $G$-crossed product}.

In the following, we will only consider unital algebras $A$ over a field $F$ which are  finite-dimensional
  without explicitly saying so.
We denote the set of invertible elements of $A$ by $A^{\times}$.

\section{Cyclic and crossed product subalgebras of central simple algebras}\label{sec:aux}

In this section, let $M/F$ be a field extension of degree $n$, and
$G={\rm Aut}_F(M)$ the group of automorphisms of $M$ which fix the
elements of $F$.  Let $A$ be a central simple algebra of degree $n$
over $F$ and suppose that $M$ is contained in $A$, i.e. is a maximal subfield of $A$.

The  results in this section are stated for central division
algebras $A$ over $F$ for instance in \cite{P66}, and none of them are proved there.
 We generalize them  to any central simple algebra $A$ with a maximal
subfield $M$ as above, so that \cite[(26)]{P66} which is well known for Galois extensions $M/F$ becomes:

\begin{lemma} \label{lem:Petit (26)}
(i) For any $\sigma \in G$ there exists an invertible $x_{\sigma} \in
A $ such that
 the inner automorphism
 $$I_{x_{\sigma}}: A \rightarrow A, \ y \mapsto x_{\sigma} y x_{\sigma}^{-1}$$
 restricted to $M$ is $\sigma$.
\\ (ii) Given any $\sigma \in G$, we have
$$\{ x \in A^\times  \ | \ I_{x} \vert_M = \sigma \} =
M^{\times}x_{\sigma}.$$
 (iii) The set of cosets $\{ M^{\times} x_{\sigma} \ \vert \ \sigma \in G\}$ together with the multiplication given by
\begin{equation} \label{eqn:M^X x_sigma M^times x_tau = M^times x_sigma tau}
M^{\times} x_{\sigma} M^{\times} x_{\tau} = M^{\times} x_{\sigma
\tau}
\end{equation}
is a group isomorphic to $G$, where $\sigma$ and $M^{\times}
x_{\sigma}$ correspond under this isomorphism.
\end{lemma}

\begin{proof}
 (i) By the Theorem of Skolem-Noether,
there exists $x_{\sigma} \in A^\times $ such that $I_{x_{\sigma}} \vert_M = \sigma$.
\\ (ii) We have
$$I_{(m x_{\sigma})}(y) = (m x_{\sigma}) y (m x_{\sigma})^{-1} =
(m x_{\sigma}) y (x_{\sigma}^{-1} m^{-1}) = m \sigma(y) m^{-1} =
\sigma(y),$$ for all $m, y \in M^{\times}$, and thus $M^{\times}
x_{\sigma} \subset \{ x \in A^\times  \ | \ I_{x} \vert_M =
\sigma \}$.

Suppose $u \in \{ x \in A^\times  \ \vert \ I_{x} \vert_M =
\sigma \}$. As $u$ and $x_{\sigma}$ are invertible, we can write
$u = v x_{\sigma}$ for some $v \in A^\times $. We still have to prove
that $v \in
M^{\times}$.
We have $$\sigma(y) = I_{u}(y) = (v x_{\sigma}) y (v
x_{\sigma})^{-1} = v x_{\sigma} y x_{\sigma}^{-1} v^{-1} = v
\sigma(y) v^{-1},$$ for all $y \in M$, and so $\sigma(y) v = v
\sigma(y)$ for all $y \in M$, that is $m v= v m$ for all $m \in M$
since $\sigma$ is bijective. Therefore $v$ is contained in the
centralizer of $M$ in $A$, which is equal to $M$
because
$M$ is a maximal subfield of $A$.
\\ (iii) Let $s=m_1 x_{\sigma} m_2 x_{\tau} \in M^{\times} x_{\sigma} M^{\times} x_{\tau}$ for some
$m_1, m_2 \in M^{\times}$, $\sigma, \tau \in G$. Then
\begin{align*}
I_{s} (y)&= (m_1 x_{\sigma} m_2
x_{\tau}) y (m_1 x_{\sigma} m_2 x_{\tau})^{-1} \\ &= m_1 x_{\sigma}
m_2 (x_{\tau} y x_{\tau}^{-1}) m_2^{-1} x_{\sigma}^{-1} m_1^{-1} \\
&= m_1 x_{\sigma} (m_2 \tau(y) m_2^{-1}) x_{\sigma}^{-1} m_1^{-1} \\
&= m_1 (x_{\sigma} \tau(y) x_{\sigma}^{-1}) m_1^{-1} \\ &= m_1
\sigma(\tau(y)) m_1^{-1} \\ &= \sigma(\tau(y)) = \sigma \tau (y),
\end{align*}
for all $y \in M$, which means $I_{s}$ restricts to $\sigma \tau$ on $M$. Therefore $m_1
x_{\sigma} m_2 x_{\tau} \in M^{\times} x_{\sigma \tau}$ by (ii) and
so $M^{\times} x_{\sigma} M^{\times} x_{\tau} \subseteq M^{\times}
x_{\sigma \tau}$. In particular, we get $x_{\sigma}x_{\tau} = m
x_{\sigma \tau}$ for some $m \in M^{\times}$. Thus $$l x_{\sigma\tau}
= l m^{-1} x_{\sigma}x_{\tau} \in M^{\times} x_{\sigma}M^{\times}
x_{\tau}$$ for all $l \in M^{\times}$, i.e. $M^{\times} x_{\sigma
\tau} \subseteq M^{\times} x_{\sigma} M^{\times} x_{\tau}$, and hence
$M^{\times} x_{\sigma \tau} = M^{\times} x_{\sigma} M^{\times}
x_{\tau}$.

Finally, the map $\{ M^{\times} x_{\sigma} \ \vert \ \sigma \in G \}
\rightarrow G, \ M^{\times} x_{\sigma} \mapsto \sigma,$ is clearly
bijective and is multiplicative by \eqref{eqn:M^X x_sigma M^times
x_tau = M^times x_sigma tau} which yields the assertion.
\end{proof}

The following  generalizes \cite[(27)]{P66} to central simple
algebras with a maximal subfield $M$ as above. The result was
again  only stated for division algebras and also not in terms of crossed product algebras:

\begin{theorem} \label{lem:Petit (27)}
 (i) $A$ contains a subalgebra $M(G)$ which is a crossed product algebra $(M,G,\mathfrak{a})$ of degree
 $|G|$ over ${\rm Fix}(G)$ with maximal subfield $M$.
\\ (ii)  $A=M(G)$ if and only if $M$ is a Galois field extension of $F$.
In that case,  $A$ is a $G$-crossed product algebra over $F$.
\\ (iii)  For any subgroup $H$ of  $G$, there is a subalgebra $M(H)$ of both $M(G)$  and $A$
which is a $H$-crossed product algebra of degree $|H|$ over ${\rm Fix}(H)$  with maximal subfield $M$.
\end{theorem}

\begin{proof}
(i)  There is an $F$-subalgebra $M(G)$ of $A$ admitting a basis $\{ x_{\sigma} \ \vert \ \sigma \in G \}$ as a
 vector space over $M$:
Let $M(G)$ denote the subset of $A$ which is generated as an
$M$-vector space by $\{ x_{\sigma} \ \vert \ \sigma \in G \}$. Note
that $1 \in M^{\times} x_{id} \subset M(G)$ and also $x_{\sigma}
x_{\tau} \in M^{\times} x_{\sigma \tau}$ for all $\sigma, \tau \in G$
by Lemma \ref{lem:Petit (26)}. In particular, there exist
$a_{\sigma,\tau} \in M^{\times}$ such that
\begin{equation} \label{eqn:M(G) rules (1)}
x_{\sigma} x_{\tau} = a_{\sigma,\tau} x_{\sigma \tau}
\end{equation}
holds for all $\sigma, \tau \in G$. Therefore $M(G)$ is closed under
multiplication, contains the identity, and can easily be seen to be
an $F$-subalgebra of $A$.

Furthermore, $\sigma(m) =
I_{x_{\sigma}}(m) = x_{\sigma} m x_{\sigma}^{-1}$ for all $m \in
M$ by Lemma \ref{lem:Petit (26)} which yields
\begin{equation} \label{eqn:M(G) rules (2)}
x_{\sigma} m = \sigma(m) x_{\sigma},
\end{equation}
for all $m \in M, \sigma, \tau \in G$.

The set $\{ x_{\sigma} \ \vert \ \sigma \in G \}$ is
linearly independent over $M$:
Suppose
\begin{equation} \label{eqn:linearly independent Petit (27)}
\sum_{\sigma \in G} m_{\sigma} x_{\sigma} = 0
\end{equation}
for some $m_{\sigma} \in M$, not all $0$, where the sum
\eqref{eqn:linearly independent Petit (27)} is chosen so that the
number of non-zero $m_{\sigma}$ is minimal. Let $\tau \in G$ be such
that $m_{\tau} \neq 0$, then
\begin{equation} \label{eqn:linearly independent Petit (27) 2}
0 = \big( \sum_{\sigma \in G} m_{\sigma} x_{\sigma} \big) m - \tau(m)
\big( \sum_{\sigma \in G} m_{\sigma} x_{\sigma} \big) = \sum_{\sigma
\in G} m_{\sigma}(\sigma(m) - \tau(m)) x_{\sigma},
\end{equation}
for all $m \in M$ by \eqref{eqn:M(G) rules (2)} and
\eqref{eqn:linearly independent Petit (27)}. The coefficient of
$x_{\tau}$ in \eqref{eqn:linearly independent Petit (27) 2} is $0$,
so by the minimality of \eqref{eqn:linearly independent Petit (27)}
we obtain $$m_{\sigma}(\sigma(m) - \tau(m)) = 0$$ for all $\sigma \in
G$. This means $\sigma = \tau$ for all $\sigma \in G$ with
$m_{\sigma} \neq 0$, a contradiction, so we proved linear independency.

$M/\mathrm{Fix}(G)$ is a Galois field extension of degree $| G |$,
and the associativity of $M(G)$ implies in particular
$(x_{\sigma} x_{\tau}) x_{\rho} = x_{\sigma} (x_{\tau} x_{\rho})$ for all $\sigma, \tau, \rho \in G$. This means
$$a_{\sigma,\tau} a_{\sigma \tau, \rho} = a_{\sigma, \tau \rho} \sigma(a_{\tau, \rho}),$$
for all $\sigma, \tau, \rho \in G$. Therefore the constants $a_{\sigma,\tau}$ define a factor set
$$\mathfrak{a}: G \times G \rightarrow M^{\times}, \ (\sigma,\tau) \mapsto a_{\sigma,\tau}$$
 of $G$ and hence
 $M(G)$ is the $G$-crossed product algebra $(M,G,\mathfrak{a})$ over $\mathrm{Fix}(G)$ of degree $|G|$.
\\ (ii) We have $[M:F] = n$ and $A$ has dimension
$n^2$ over $F$. If $M$ is not a Galois extension of $F$, then
$\vert G \vert < n$ and thus $\{
x_{\sigma} \ \vert \ \sigma \in G \}$ cannot be a  set of generators
for $A$ as a vector space over $M$. Conversely, if $M/F$ is a Galois
extension, then $\vert G \vert = n$ and since $\{ x_{\sigma} \ \vert
\ \sigma \in G \}$ is linearly independent over $M$, counting
dimensions yields $M(G) = A$. The rest of the assertion is trivial.
\\ (iii)   For any subgroup $H$ of  $G$, there is an $F$-subalgebra $M(H)$ of $M(G)$ with basis
$\{ x_{\sigma} \ \vert \ \sigma \in H \}$ as a  vector space over $M$
and multiplication in $M(H)$ defined by constants $a_{\sigma,\tau}
\in M^{\times}$ for all $m \in M$, $\sigma, \tau \in H$ according to
the rules in (i):
clearly $M(H)$ is closed under multiplication since if $x_{\sigma}, x_{\tau} \in M(H)$,
then $\sigma, \tau \in H$, hence also $\sigma \tau \in H$ and so
$x_{\sigma \tau} \in M(H)$. Additionally $1 \in M^{\times} x_{id}
\subset M(H)$, and thus  $M(H)$ is a subalgebra of $M(G)$.
$M/\mathrm{Fix}(H)$ is a Galois field extension of degree $| H|$ and with the same argument as in
the proof of (i) thus  $M(H)$ is a $H$-crossed product algebra over $\mathrm{Fix}(H)$ of degree $|H|$.
\end{proof}

More precisely, a closer look at the above proof reveals:

\begin{lemma} \label{prop:Petit (28)}
(i) For any subgroup $H$ of $G$,
$M(H)$ is a $H$-crossed product algebra over its center with
$$M(H)=(M,H,\mathfrak{a}_H),$$
 where $\mathfrak{a}_H$ denotes the
factor set $\mathfrak{a}$ of the crossed product algebra $M(G)=(M,G,\mathfrak{a})$, restricted to the elements in $H$.
\\ (ii) If $H$ is a cyclic subgroup of $G$ of order $h>1$ generated by
$\sigma\in G$, then there exists $c \in {\rm Fix}(\sigma)^\times$ such that
$$M(H)\cong M[t;\sigma]/M[t;\sigma](t^h - c)$$
is a cyclic algebra of degree $h$ over ${\rm Fix}(\sigma)$ and an $F$-subalgebra of $A$.
\end{lemma}

\begin{proof}
(i) is trivial.
\\ (ii) $M(H)$ is a $H$-crossed product algebra $(M,H,\mathfrak{a}_H)$ over $\mathrm{Fix}(\sigma)$ of degree
$h$ by Theorem \ref{lem:Petit (27)} and $H$ is a cyclic group, therefore $M(H)$ is a cyclic algebra over
$\mathrm{Fix}(\sigma)$ of
degree $h$, i.e. there exists $c \in {\rm Fix}(\sigma)^\times$ such that
$M(H)\cong M[t;\sigma]/M[t;\sigma](t^h - c)$, e.g. see \cite[p. 19]{J96}.
\end{proof}

We conclude that even if a central division algebra $A$ over $F$ is a noncrossed product,
if $A$ contains a maximal field extension $M$  with a non-trivial $\sigma\in G = {\rm Aut}_F(M)$ of order $h$,
then it contains a cyclic division algebra of degree $h$ (though generally not with center $F$):

\begin{theorem} \label{cor:importantI}
Let $A$ be a central division algebra over $F$ with maximal subfield $M$ and non-trivial $\sigma\in G = {\rm Aut}_F(M)$
of order $h$. Then  $A$ contains the cyclic division algebra
$$(M/{\rm Fix} (\sigma), \sigma, c)=M[t;\sigma]/M[t;\sigma](t^h - c)$$
 of degree $h$ over ${\rm Fix} (\sigma)$ as an $F$-subalgebra.
\end{theorem}

This generalizes  \cite[(28)]{P66} to central simple algebras with a maximal subfield $M$ such that $G={\rm Aut}_F(M)$ is not trivial.

\begin{remark}
The question when a central division algebra $A$ over $F$ has a cyclic subalgebra of prime degree was recently
 raised in \cite[Question 1]{M}. If $F$ is a Henselian field such that $\overline F$
is a global field, and $A$  is an central division algebra over $F$ such that ${\rm char}(\overline F)$ does not divide
$\deg (A),$ then $A$ contains a cyclic division algebra of prime degree  \cite[Theorem 3]{M}.

By Theorem \ref{cor:importantI}, any central division algebra $A$ over $F$ containing a maximal subfield $M$ with some
 $\sigma\in G = {\rm Aut}_F(M)$ of prime order $p$ contains a cyclic division algebra of prime degree $p$.
\end{remark}

 A central division algebra of prime degree over $F$ is cyclic if and only if it has a cyclic subalgebra
of prime degree (not necessarily with center $F$) \cite[p.~2]{M}. Theorem \ref{cor:importantI}  yields
the following observations:

\begin{corollary} \label{cor:importantII}
Let $A$ be a central division algebra over $F$.
\\ (i)  If $A$ has prime degree
then either $A$ is a cyclic algebra or each of its maximal subfields $M$ has trivial automorphism group $ {\rm Aut}_F(M)$.
\\
(ii) Suppose $A$ contains a maximal subfield $M$ such that $G={\rm Aut}_F(M)$
is non-trivial. Then $A$ contains the $G$-crossed product division
algebra $M(G)=(M,G,\mathfrak{a})$ of degree $| G |$ over ${\rm Fix}(G)$
as a subalgebra.
\end{corollary}

\begin{proof}
(i) If $G= {\rm Aut}_F(M)$ is non-trivial, then $A$ contains the cyclic algebra $M(G)$ of degree $|G|\leq p$ over
$F_0={\rm Fix}(G)$ as subalgebra. Since $M/F_0$ is a maximal subfield of $M(G)$, it also has degree $|G|$.
Looking at the possible degrees of the intermediate field extensions of $M/F$ we have
$[M:F_0]=1$ or $[M:F_0]=p$, so $|G|=1$ or $|G|=p$.
If $|G|=p$ then $A$ is a cyclic algebra. Hence if $A$ is not cyclic then
each of its maximal subfields $M$ must have  trivial automorphism group ${\rm Aut}_F(M) $.
\\ (ii) is trivial.
\end{proof}

\section{Central simple algebras containing maximal subfields with a solvable $F$-automorphism group} \label{sec:main}

Let $G$ be finite a solvable group, i.e. there exists a chain of subgroups
\begin{equation} \label{eqn:Subnormal series}
\{ 1 \} = G_0 \leq G_1 \leq \ldots \leq G_k = G,
\end{equation}
such that $G_{j}$ is normal in $G_{j+1}$ and $G_{j+1}/G_{j}$ is cyclic of
prime order $q_{j}$ for all $j \in \{ 0, \ldots,k-1 \}$, that is
\begin{equation} \label{eqn:proof petit (29) 1}
G_{j+1}/G_{j}  = \{ G_{j},G_{j} \sigma_{j+1}, \dots \},
\end{equation}
for some
$\sigma_{j+1} \in G_{j+1}$.
 Lemma \ref{lem:Petit (26)}, Theorem \ref{lem:Petit (27)} and Corollary \ref{prop:Petit (28)} yield
the following generalization of \cite[(29)]{P66}, which only claims the result for central division algebras over $F$:

\begin{theorem} \label{thm:Petit (29)}
Let $M/F$ be a field extension of degree $n$ with non-trivial solvable $G = {\rm
Aut}_F(M)$, and $A$  a central simple algebra of degree
$n$ over $F$ with maximal subfield $M$. Then there exists a chain of
subalgebras
\begin{equation} \label{eqn:Petit (29) chain of subalgebras}
M = A_0 \subset A_1 \subset \ldots \subset A_k = M(G) \subset A,
\end{equation}
of $A$ which are  $G_i$-crossed product algebras over $Z_i={\rm Fix}(G_i)$ and
where
\begin{equation} \label{eqn:Petit (29) chain of subalgebras 2}
A_{i+1} \cong A_i[t_i;\tau_i]/A_i[t_i;\tau_i](t_i^{q_i} - c_i)
\end{equation}
for all $i \in \{ 0, \ldots, k-1 \}$, such that
\\ (i)  $q_i$ is the prime order of the
factor group $G_{i+1}/G_i$ in the chain of normal subgroups
\eqref{eqn:Subnormal series},
\\ (ii) $\tau_i$ is an $F$-automorphism of $A_i$ of inner order $q_i$ which restricts to the
automorphism $\sigma_{i+1} \in G_{i+1}$ that generates $G_{i+1}/G_i$, and
\\ (iii) $c_i \in {\rm Fix}(\tau_i)$ is
invertible.
\end{theorem}

Note that the inclusion $M(G) \subset A$ in \eqref{eqn:Petit (29)
chain of subalgebras}
 is an equality if and only if $M/F$ is a Galois extension by Theorem \ref{lem:Petit (27)}, i.e.
  if and only if $A$ is a $G$-crossed product algebra.

\begin{proof}
Define $A_i= M(G_i)$ for all $i \in \{ 1, \ldots, k \}$. $A_i$ is a $G_i$-crossed product algebra over ${\rm Fix}(G_i)$
by Theorem \ref{lem:Petit (27)}.

 $G_1/G_0 \cong
G_1$ is a cyclic subgroup of $G$ of order $q_0$ generated by some
$\sigma_1 \in G$. Let $\tau_0 = \sigma_1$, then there exists $c_0 \in {\rm Fix}(\tau_0)$ such
that $A_1 = M(G_1)$ is $F$-isomorphic to
$$M[t_0;\tau_0]/M[t_0;\tau_0](t_0^{q_0} - c_0),$$
by Corollary \ref{prop:Petit (28)}, which is a cyclic algebra of prime degree $q_0$ over $\mathrm{Fix}(\tau_0)$.

Now $G_1 \triangleleft G_2$ and $G_2/G_1$ is cyclic of prime order
$q_1$ with
\begin{equation}
G_2/G_1 = \{ G_1, G_1 \sigma_2, \ldots, G_1 \sigma_2^{q_1-1} \}
\end{equation}
for some $\sigma_2 \in G_2$. Hence we can write
$G_2 = \{ h
\sigma_2^i \ \vert \ h \in G_1, 0 \leq i \leq q_1-1 \}$ and thus $A_2
= M(G_2)$ has a basis $$\{ x_{h \sigma_2^j} \ \vert \ h \in G_1, \ 0\leq j \leq q_1-1 \},$$
as an $M$-vector space. Recall $$M^{\times} x_{h
\sigma_2^j} = M^{\times} x_{h} x_{\sigma_2^j} = M^{\times} x_{h}
x_{\sigma_2}^j$$
 for all $h \in G_1$ by Lemma \ref{lem:Petit (26)}, and $\{ 1, x_{\sigma_2}, \ldots, x_{\sigma_2}^{q_1-1} \}$
 is a basis for $A_2$ as a left $A_1$-module, i.e.
\begin{equation} \label{eqn:proof petit (29) 2}
A_2 = A_1 + A_1 x_{\sigma_2} + \ldots + A_1 x_{\sigma_2}^{q_1-1}.
\end{equation}
 We have $G_2 G_1 = G_1 G_2$ as
$G_1$ is normal in $G_2$ and so
for every $h \in G_1$, we get
$\sigma_2 h = h^\prime\sigma_2$ for some $h^\prime \in G_1$. Choose the basis
$\{ x_{h}\,|\, h\in G_1\}$
of
$A_1$ as a vector space over $M$.
 By
\eqref{eqn:M(G) rules (1)} we obtain
\begin{equation} \label{eqn:proof petit (29) 3}
x_{\sigma_2} x_{h} = a_{\sigma_2,h} x_{\sigma_2 h} =
a_{\sigma_2,h} x_{h^\prime\sigma_2} = a_{\sigma_2,h}
(a_{h^\prime,\sigma_2})^{-1} x_{h^\prime} x_{\sigma_2}.
\end{equation}

Recall $x_{\sigma_2}\in A^\times$  by Lemma \ref{lem:Petit (26)}. The inner automorphism
 $$\tau_1: A \rightarrow A,\ z \mapsto x_{\sigma_2} z x_{\sigma_2}^{-1}$$
restricts to $\sigma_2$ on $M$. Moreover,
\begin{equation} \label{eqn:proof petit (29) 4}
\begin{split}
\tau_1(x_{h_j}) &= x_{\sigma_2} x_{h_j} x_{\sigma_2}^{-1} =
 a_{\sigma_2,h_j} (a_{h_j^\prime,\sigma_2})^{-1} x_{h_j^\prime} x_{\sigma_2} x_{\sigma_2}^{-1} \\
 &= a_{\sigma_2,h_j} (a_{h_j^\prime,\sigma_2})^{-1} x_{h_j^\prime} \in A_1,
\end{split}
\end{equation}
for all $h_j \in G_1$, i.e. $\tau_1 \vert_{A_1}(y) \in A_1$ for all $y \in A_1$ and so $\tau_1 \vert_{A_1}$ is an
$F$-automorphism of $A_1$. Moreover,
 $$x_{\sigma_2} x_{h} = \tau_1 \vert_{A_1}(x_{h})
x_{\sigma_2},$$ for all $h \in G_1$ by \eqref{eqn:proof petit (29)
3}, \eqref{eqn:proof petit (29) 4}, and $$x_{\sigma_2} m =
\sigma_2(m) x_{\sigma_2} = \tau_1 \vert_{A_1}(m) x_{\sigma_2},$$ for
all $m \in M$.
We conclude that
\begin{equation} \label{eqn:proof petit (29) 6}
x_{\sigma_2} y = \tau_1 \vert_{A_1}(y) x_{\sigma_2}
\end{equation}
for all $y \in A_1$. Define $c_1 = x_{\sigma_2}^{q_1}$, then
$\sigma_2^{q_1} \in G_1$ by \eqref{eqn:proof petit (29) 1} which
implies $c_1 \in A_1$.  Furthermore $c_1$ is invertible since
$x_{\sigma_2}$ is invertible. Also,
$$\tau_1 \vert_{A_1}(c_1) = x_{\sigma_2}
x_{\sigma_2}^{q_1} x_{\sigma_2}^{-1} = c_1$$ which means $c_1 \in {\rm
Fix}(\tau_1 \vert_{A_1})^\times$. Notice $$x_{\sigma_2^{-q_1}}
x_{\sigma_2^{q_1}} = a_{\sigma_2^{-q_1},\sigma_2^{q_1}} x_{{\rm
id}}\in M^{\times},$$ therefore $c_1^{-1} = x_{\sigma_2^{q_1}}^{-1}
\in M^{\times}x_{\sigma_2^{-q_1}} \in A_1$ as $\sigma_2^{-q_1} \in
G_1$. Hence $\tau_1 \vert_{A_1}$ has inner order $q_1$, since indeed
$$(\tau_1 \vert_{A_1})^{q_1}:A_1 \rightarrow A_1, z \mapsto c_1 z
c_1^{-1},$$
 is an inner automorphism.

Consider the  algebra $$B_2 = A_1[t_1;\tau_1 \vert_{A_1}]
/A_1[t_1;\tau_1 \vert_{A_1}](t_1^{q_1} - c_1)$$
with center
\begin{align*}
C(B_2) & \supset \{ b \in A_1 \ \vert \ bh = hb \text{ for
all } h \in B_2 \} = C(A_1) \cap {\rm Fix}(\tau_1) \supset
F.
\end{align*}
By \eqref{eqn:proof petit (29) 2} and \eqref{eqn:proof petit
(29) 6}, the $F$-linear map $$\phi: A_2 \rightarrow B_2, \ y x_{\sigma_2}^i
\mapsto y t_1^i \qquad \qquad (y \in A_1),$$ is
 an isomorphism.  In addition, by a straightforward calculation we have
$$\phi \big( (y x_{\sigma_2}^i)(z x_{\sigma_2}^j) \big)=\phi(y
x_{\sigma_2}^i)  \phi(z x_{\sigma_2}^j)$$
for all $y, z \in A_1$, $i, j \in \{ 0, \ldots, q_1-1 \}$, so
$\phi$ is also multiplicative, thus an $F$-algebra isomorphism. Continuing in this manner for $G_2
\triangleleft G_3$ etc. yields the assertion.
\end{proof}

For a subset $B$ in $ A$, let  ${\rm Cent}_{A}(B)$ denote the centralizer of $B$ in $A$.
Then  the algebras $A_i$ are the centralizers of $\mathrm{Fix}(G_{i})$ in $A_{i+1}$:

\begin{corollary} \label{cor:Centralizer of Z_i in A_i+1}
Let $M/F$ be a field extension of degree $n$ with non-trivial solvable $G = {\rm Aut}_F(M)$
with normal series \eqref{eqn:Subnormal series}, and $A$  a central simple algebra of degree
$n$ over $F$ with maximal subfield $M$. Then
 $$A_i = {\rm Cent}_{A_{i+1}}(\mathrm{Fix}(G_i))$$
  for all $i \in \{ 0, \ldots, k-1 \}$ where
 $A_i = M(G_i)$ are as in Theorem \ref{thm:Petit (29)}.
\end{corollary}

\begin{proof}
Clearly $A_i \subset {\rm Cent}_{A_{i+1}}(\mathrm{Fix}(G_i))$ for all
$i \in \{ 0, \ldots, k-1 \}$ because $A_i \subset A_{i+1}$ and $C(A_i) = \mathrm{Fix}(G_i)$.
To prove ${\rm Cent}_{A_{i+1}}(\mathrm{Fix}(G_i)) \subset A_i$,
let $x = \sum_{\sigma \in G_{i+1}} m_{\sigma} x_{\sigma} \in {\rm Cent}_{A_{i+1}}(\mathrm{Fix}(G_i))$ for some
$m_{\sigma} \in M$. If $\tau \in G_{i+1} \setminus G_i$ is such that $m_{\tau} \neq 0$, then $m_{\tau} x_{\tau} z = z m_{\tau} x_{\tau}$ for all $z \in \mathrm{Fix}(G_i)$, that is, $\tau(z) = z$ for all $z \in \mathrm{Fix}(G_i)$.

Now $\mathrm{Fix}(G_{i+1})$ is properly contained in $\mathrm{Fix}(G_i)$, therefore
$\tau(z) \neq z$ for all $z \in \mathrm{Fix}(G_i) \setminus \mathrm{Fix}(G_{i+1})$, a contradiction.
This implies $x = \sum_{\sigma \in G_i} m_{\sigma} x_{\sigma}$ as required.
\end{proof}

\begin{corollary} \label{cor:importantIII}
Let $A$ be a central division algebra over $F$ containing a
maximal subfield $M$ with non-trivial solvable $G = {\rm Aut}_F(M)$. Then:
\\ (i) $A$
 contains the cyclic division algebra $(M/{\rm Fix} (\sigma_1), \sigma_1, c_0)$ of prime degree $q_0$ over
  ${\rm Fix} (\sigma_1)$ as a subalgebra.
  \\ (ii) There is a non-central element $t_0 \in A$ such that $t_0^{q_0}\in {\rm Fix} (\sigma_1)^\times$
  and $t_0^{m}\not\in {\rm Fix} (\sigma_1)$ for all $1\leq m<q_0$.
\\
 Here, $q_0$ is the order of the cyclic subgroup  $G_1$ of the normal subseries (\ref{eqn:Subnormal series}) of $G$.
\end{corollary}

Additionally, we obtain the following straightforward observations:

\begin{corollary} \label{cor:dimension and center of M_i}
Let $A_i$ be as in
 (\ref{eqn:Petit (29) chain of subalgebras}) of Theorem \ref{thm:Petit (29)}, $i \in \{ 0, \ldots, k \}$.
\\ (i)  $A_i$ is a generalized cyclic algebra over  $Z_i={\rm Fix}(G_i) $ of degree
 $${\rm deg}(A_{i-1})q_{i-1}=\prod_{l=0}^{i-1} q_l$$
and
\begin{equation} \label{eqn:dimension and center of M_i 1}
M = Z_0 \supset \ldots \supset Z_{k-1} \supset Z_k \supset F.
\end{equation}
 (ii) $M/Z_i$ is  a Galois extension and $M$ is a maximal subfield of $A_i$.
\end{corollary}

\begin{proof}
(i) and (ii):
  $A_i$ is a generalized cyclic
algebra as defined in \ref{sec:gca}. Since $G_{i-1} \leq G_i$ we have
$$Z_i = {\rm Fix}(G_i) \subset {\rm Fix}(G_{i-1}) = Z_{i-1},$$
 for all $i \in \{1, \ldots, k \}$.
 We know that  $A_i$ has ${\rm deg}(A_{i})={\rm deg}(A_{i-1})q_{i-1}$ over its center. By induction we obtain the assertion.
\\ (ii) is trivial by Theorem \ref{thm:Petit (29)}.
\end{proof}

Hence even if a central division algebra $A$ over $F$ is a noncrossed product,
if $A$ contains a maximal field extension $M$  with non-trivial solvable $G = {\rm Aut}_F(M)$
then it contains a chain of generalized cyclic division algebras:

\begin{corollary} \label{cor:important}
Let $M/F$ be a field extension of degree $n$ with non-trivial solvable $G = {\rm Aut}_F(M)$, and $A$
a central division algebra over $F$ with maximal subfield $M$. Then  $A$
 contains a chain of generalized cyclic division algebras $A_i$  over intermediate fields $Z_i={\rm Fix}(G_i)$
  of $M/F$ as in (\ref{eqn:Petit (29) chain of subalgebras}).
 Here, $q_i$ is the order of the cyclic factor group  $G_{i+1}/G_i$ of the normal subseries (\ref{eqn:Subnormal series}) of $G$.
\end{corollary}

If $A$ is a division algebra in the above setup then $G$ solvable implies that $A^\times$ contains an solvable subgroup:

\begin{lemma}
Suppose $A$ is a  central division algebra over $F$. If $A$ contains a maximal subfield $M/F$ with solvable
 $G={\rm Aut}_F(M)$ then
$A^\times$ contains an solvable subgroup. If $M/F$ is Galois, i.e. $A$ a $G$-crossed product algebra, then
this solvable subgroup is irreducible.
\end{lemma}

\begin{proof}
As noted in \cite[Lemma 1]{KMH} (where $M/F$ is Galois, but the argument is the same),
$N=\bigcup_{\sigma\in G} M^\times x_\sigma \subset A^\times$ and it is easy to see that $M^\times$ is a normal
subgroup of $N$, and that $N$ is the normalizer of $M^\times$ in $A^\times.$
Therefore $N/M^\times\cong G$ as in Lemma \ref{lem:Petit (26)} (iii) and if $G$ is solvable as assumed in later sections,
we see that in fact $N$ is a solvable subgroup of $A^\times.$
Since $N=\bigcup M^\times x_\sigma$, if $M/F$ is Galois, i.e. $A$ a $G$-crossed product algebra, then $N$ is irreducible,
i.e. the $F$-algebra generated by elements of $N$, $F[N]$, is $A$ by \cite[Lemma 1]{KMH}.
\end{proof}

Our next result generalizes \cite[(9)]{P68} and characterizes all
the algebras with a maximal subfield $M/F$ that have a solvable automorphism group
 $G = \mathrm{Aut}_{F}(M)$ via generalized cyclic algebras:

\begin{theorem} \label{thm:Petit (29) Generalised}
Let $M/F$ be a field extension of degree $n$ with non-trivial $G = \mathrm{Aut}_{F}(M)$, and $A$ be a
central simple algebra over $F$ with maximal subfield $M$. Then $G$ is solvable if
 there exists a chain of  subalgebras
\begin{equation} \label{eqn:Petit (29) chain of subalgebras Galois}
M = A_0 \subset A_1 \subset \ldots \subset A_k  \subset A
\end{equation}
of $A$ which all have maximal subfield $M$,  where $A_k$ is a $G$-crossed product algebra over ${\rm Fix}(G)$, and where
\begin{equation} \label{eqn:Petit (29) chain of subalgebras Galois 2}
A_{i+1} \cong A_i[t_i;\tau_i]/A_i[t_i;\tau_i](t_i^{q_i} - c_i),
\end{equation}
for all $i \in \{ 0, \ldots, k-1 \}$,  with
\begin{itemize}
\item[(i)] $q_i$ a prime,
\item[(ii)] $\tau_i$ an $F$-automorphism of $A_i$ of inner order $q_i$ which restricts to an automorphism $\sigma_{i+1} \in G$, and
\item[(iii)] $c_i \in \mathrm{Fix}(\tau_i)^\times$.
\end{itemize}
\end{theorem}

\begin{proof}
Suppose there exists a  chain of algebras $A_i$, $i \in \{ 0, \ldots, k \}$
satisfying the above assumptions. Put  $G_k=G$.
Since each $A_i$ has center $Z_i = Z_{i-1} \cap \mathrm{Fix}(\tau_{i-1})$, so that by induction
$$Z_i={\rm Fix}(\tau_0) \cap {\rm Fix}(\tau_1)\cap \dots\cap{\rm Fix}(\tau_{i-1})\supset F,$$
  $M/Z_i$ is a Galois extension contained in $A_i$. Put $G_i={\rm Gal}(M/Z_i)$, then
each $A_i$ is a $G_i$-crossed product algebra.
In particular, $G_i$ is a subgroup of $G_{i+1}$.

We use induction to prove that  each $G_i$, thus $G$, is a solvable group.

For $i = 1$,
$$A_1 \cong M[t_0;\sigma_1]/M[t_0;\sigma_1](t_0^{q_0}-c_0)$$
is a cyclic algebra of degree $q_0$ over $\mathrm{Fix}(\sigma_1)$.
 $G_1 = <\sigma_1>$ is a cyclic group of prime order $q_0$ and therefore solvable.

We assume as induction hypothesis that if there exists a  chain
$$M = A_0 \subset \ldots \subset A_j $$
of  algebras such that \eqref{eqn:Petit (29) chain of
subalgebras Galois 2} holds for all $i \in \{ 0, \ldots, j-1 \}$, $j \geq 1$,
then $G_j$ is solvable.

For the induction step we take a chain of  algebras
$M = A_0 \subset \ldots \subset A_j\subset A_{j+1},$
$$A_{i+1} \cong A_i[t_i;\tau_i]/A_i[t_i;\tau_i](t_i^{q_i} -c_i)$$
 where $\tau_i$ is an automorphism
of $A_i$ of inner order $q_i$ which induces an automorphism  $\sigma_{i+1}
\in G$, $c_i \in {\rm Fix}(\tau_i)$ is invertible and $q_i$ is prime, for all $i \in \{ 0, \ldots, j \}$.
  By the induction hypothesis, $G_j$ is a solvable group.

We show that $G_{j+1}$ is solvable: $t_j$ is an invertible element of
$$ A_{j+1}\cong A_j[t_j;\tau_j]/A_j[t_j;\tau_j](t_j^{q_j} - c_j),$$
 with inverse $c_j^{-1}t_j^{q_j-1}$.

$A_{j}$ is a $G_{j}$-crossed product algebra over $Z_{j}$ with maximal subfield
$M$. The $F$-automorphism $\tau_j$ on $A_j$
satisfies $t_j l = \tau_j(l) t_j$ for all $l \in A_j$ which implies
the inner automorphism $$I_{t_j}: A \rightarrow A, \ d \mapsto
t_j d t_j^{-1}$$
restricts to $\tau_j$ on $A_j$ and
to $\sigma_{j+1}$ on $M$.

 For any $\sigma \in G$ there exists an invertible $x_{\sigma} \in
A $ such that
 the inner automorphism
 $$I_{x_{\sigma}}: A \rightarrow A, \ y \mapsto x_{\sigma} y x_{\sigma}^{-1}$$
 restricted to $M$ is $\sigma$.
 Hence we have $x_{\sigma_{j+1}} = t_j$ with $x_{\sigma_{j+1}}$ as defined in Lemma \ref{lem:Petit (26)}. We know that
$\{ 1, t_j, \ldots, {t_j}^{q_j-1} \}$ is a basis for $A_{j+1}$ as a left $A_j$-module.
By \eqref{eqn:M(G) rules (1)}  we have  $x_{\sigma_{j+1}^2} = a_1
{t_j}^2$, $x_{\sigma_{j+1}^3} =a_2 {t_j}^3, \dots$ for suitable $a_i \in M^\times$,
so that w.l.o.g. $\{ 1, x_{\sigma_{j+1}}, \ldots, x_{\sigma_{j+1}^{q_j-1}} \}$ is a
basis for $A_{j+1}$ as a left $A_j$-module.

 Since $A_j$ is a $G_j$-crossed product algebra, it has an $M$-basis
 $\{ x_{\rho} \ \vert \ \rho \in G_j \}$, and hence $A_{j+1}$
has  basis $$\{ x_{\rho} x_{\sigma_{j+1}^i} \ \vert \ \rho \in G_j, \ 0
\leq i \leq q_j - 1 \}$$ as  $M$-vector space.

Additionally,    $x_{\rho} x_{\sigma_{j+1}^i} \in M^{\times} x_{\rho
\sigma_{j+1}^i}$ by Lemma \ref{lem:Petit (26)} (iii) and thus $A_{j+1}$ has the
$M$-basis
\begin{equation}
\{ x_{\rho \sigma_{j+1}^i} \ \vert \ \rho \in G_j, \ 0 \leq i \leq q_j-1
\}.
\end{equation}
 $A_{j+1}$ is a $G_{j+1}$-crossed product algebra and thus also has the $M$-basis
 $\{ x_{\sigma} \ \vert \ \sigma \in G_{j+1} \}$. We use these two basis to show that $G_{j+1} = G_{j} <\sigma_{j+1}>$:
Write
$$x_{\rho \sigma_{j+1}^i} = \sum_{\sigma \in G_{j+1}} m_{\sigma}
x_{\sigma}$$
 for some $m_{\sigma} \in M$, not all zero. Then $$x_{\rho
\sigma_{j+1}^i} m = \sum_{\sigma \in G_{j+1}} m_{\sigma} x_{\sigma} m =
\sum_{\sigma \in G_{j+1}} m_{\sigma} \sigma(m) x_{\sigma},$$ and $$
x_{\rho \sigma_{j+1}^i} m = \rho \sigma_{j+1}^i(m) x_{\rho \sigma_{j+1}^i} =
\rho \sigma_{j+1}^i(m) \sum_{\sigma \in G_{j+1}} m_{\sigma} x_{\sigma},$$
for all $m \in M$. Let $\sigma \in G_{j+1}$ be such that $m_{\sigma}
\neq 0$, then in particular $$m_{\sigma} \sigma(m) x_{\sigma} = \rho
\sigma_{j+1}^i(m)  m_{\sigma} x_{\sigma},$$ for all $m \in M$, that is
$\sigma =  \rho \sigma_{j+1}^i$. This means that $\{  \rho \sigma_{j+1}^i \ |
\ \rho \in G_j, \ 0 \leq i \leq q_j-1 \} \subset G_{j+1}$. Both sets have the same size
so must be equal and we conclude $G_{j+1} = G_j <\sigma_{j+1}>$.

Finally we prove $G_j$ is a normal subgroup of $G_{j+1}$:  the
inner automorphism $I_{x_{\sigma_{j+1}}}$ restricts to the $F$-automorphism
$\tau_j$ of $A_j$. In particular, this implies $$x_{\sigma_{j+1}}
x_{\rho} x_{\sigma_{j+1}}^{-1} \in A_j,$$ for all $\rho \in G_j$.
Furthermore, $$x_{\sigma_{j+1} \rho \sigma_{j+1}^{-1}} \in M^{\times}
x_{\sigma_{j+1}} x_{\rho} x_{\sigma_{j+1}^{-1}} = M^{\times} x_{\sigma_{j+1}}
x_{\rho} x_{\sigma_{j+1}}^{-1} \subset A_j,$$ for all $\rho \in G_j$ by
Lemma \ref{lem:Petit (26)}.

Hence $\sigma_{j+1} \rho \sigma_{j+1}^{-1} \in
G_j$ because $A_j$ is a $G_j$-crossed product algebra. Similarly, we see $\sigma_{j+1}^r \rho \sigma_{j+1}^{-r} \in G_j$ for all
$r \in \mathbb{N}$. Let $g \in G_{j+1}$ be arbitrary and write $g = h
\sigma_{j+1}^r$ for some $h \in G_j$, $r \in \{ 0, \ldots, q_j-1 \}$ which
we can do because $G_{j+1} = G_j <\sigma_{j+1}>$. Then
\begin{align*}
g \rho g^{-1} &= ( h \sigma_{j+1}^r) \rho ( h \sigma_{j+1}^r)^{-1} = h
(\sigma_{j+1}^r \rho \sigma_{j+1}^{-r})h^{-1} \in G_j,
\end{align*}
for all $\rho \in G_j$ so $G_j$ is indeed normal.

It is well known that a group $G$ is solvable if and only if given a
normal subgroup $H$ of $G$, both $H$ and $G/H$ are solvable. It is
clear  now that $G_{j+1}/G_j$ is cyclic and hence solvable, which implies $G_{j+1}$ is
solvable as required.
\end{proof}

\section{Solvable crossed product algebras} \label{sec:crossedproduct}

We keep the assumptions from the previous section, but from now on we  focus on the case that $M/F$ is a Galois extension,
  i.e. now $A$ is a  $G$-crossed product algebra.

We obtain the next result as a special case of Theorem \ref{thm:Petit (29)}:

\begin{theorem}\label{thm:main1}
Let $M/F$ be a Galois field extension of degree $n$ with non-trivial solvable $G = {\rm
Aut}_F(M)$, and $A$  a central simple algebra of degree
$n$ over $F$ with maximal subfield $M$.
 Then $A$  is a $G$-crossed product algebra  and there exists a chain of subalgebras
$$
M = A_0 \subset A_1 \subset \ldots \subset A_k = M(G)= A,
$$
of $A$ which are generalized cyclic algebras of degree
 $\prod_{l=0}^{i-1} q_l$ over $Z_i={\rm Fix}(G_i)$ of the type
$$
A_{i+1} \cong A_i[t_i;\tau_i]/A_i[t_i;\tau_i](t_i^{q_i} - c_i)
$$
for all $i \in \{ 0, \ldots, k-1 \}$, satisfying (i), (ii), (iii) in Theorem \ref{thm:Petit (29)}. Additionally
 the following holds for all $i \in \{ 1, \ldots, k \}$:
\\ (iv)  $Z_{i-1} / Z_i$ has prime degree $q_{i-1}$.
\end{theorem}

\begin{proof}
$A_k = M(G) = A$ by Theorem \ref{lem:Petit (27)} and thus $Z_k = F$. It remains to prove
\\ (iv): $Z_{i-1} / Z_i$ is a proper field extension for all $i \in \{ 1, \ldots, k \}$ by the
Fundamental Theorem of Galois Theory, because $G_{i-1}$ is a proper
subgroup of $G_i$. We have $n = \vert G \vert = \prod_{l=0}^{k-1} q_l$ is
the decomposition of $n$ as a product of $k$ primes by
\eqref{eqn:dimension and center of M_i 1}. Also $$n = [M:F] = [M:Z_1]
\cdots [Z_{k-1}:F],$$ and so $[Z_{i-1}:Z_i] = q_{\pi(i-1)}$ for all
$i \in \{ 1, \ldots, k \}$ where $\pi$ is a permutation of $\{ 0,
\ldots, k-1 \}$. Hence $[Z_{i-1}:Z_i]$ is prime
 for all $i \in \{ 1, \ldots, k\}$.
 We now prove $q_{i-1} = q_{\pi(i-1)}$ for all $i \in \{
1, \ldots, k \}$: Suppose towards a contradiction $j \in \{ 1,
\ldots, k \}$ is such that $[Z_{j-1}:Z_j] = q_{\pi(j-1)} \neq
q_{j-1}$, where we take $j$ to be as small as possible. The dimension of $A_j$ over $Z_j$ is
$$\prod_{l=0}^{j-1} q_l^2=[M:Z_j] \vert G_j \vert=[M:Z_j]\big(\prod_{l=0}^{j-1} q_l \big).$$
Since
$$[M:Z_j]=\prod_{l=0}^{j-1} [Z_l:Z_{l+1}],$$
we obtain
$$[M:Z_j] \big( \prod_{l=0}^{j-1} q_l \big) = \big( \prod_{l=0}^{j-1}
q_{\pi(l)} \big) \big( \prod_{l=0}^{j-1} q_l \big) = \begin{cases}
\big( \prod_{l=0}^{j-2} q_l \big)^2 q_{\pi(j-1)}q_{j-1} & \text{ if }
j \geq 2 ,\\ q_{\pi(0)} q_0 & \text{ if } j=1,
\end{cases}
$$
 where we have used the minimality of $j$. Since the $q_l$
are prime and $q_{j-1} \neq q_{\pi(j-1)}$, this
implies that the dimension of $A_j$ over $Z_j$ is not a
square, a contradiction. Thus
$[Z_{i-1}:Z_i] = q_{i-1}$ for all $i \in \{1, \ldots, k \}$.
\end{proof}

In general, it is not always easy to decide if a given crossed product algebra is a division algebra or not.

\begin{theorem}\label{thm:div}
In the setup of Theorem \ref{thm:main1}, the solvable crossed product algebra $A$ is a division algebra if
and only if
\begin{equation} \label{eqn:crossed product division condition}
b \tau_i(b) \cdots \tau_i^{q_i -1}(b) \neq c_i
\end{equation}
for all $b \in A_i$ and  $i \in \{ 0, \ldots, k-1 \}$.
\end{theorem}

\begin{proof}
If $A$ is a division algebra then so are all the
subalgebras $A_{i}$,  $i \in \{ 0, \ldots, k-1 \}$. In particular,
this means that $t_i^{q_i} - c_i \in A_i[t_i;\tau_i]$ is an irreducible twisted polynomial
for all $i \in \{ 0, \ldots, k-1 \}$, i.e.
$$b \tau_i(b) \cdots
\tau_i^{q_i -1}(b) \neq c_i$$ for all $b \in A_i$ \cite[1.3.16]{J96}.

Conversely suppose \eqref{eqn:crossed product division condition}
holds for all $b \in A_i$ and  $i \in \{ 0, \ldots, k-1 \}$. We
prove by induction that then $A_i$ is a division algebra for all $i\in \{ 0, \ldots, k\}$, thus in particular
so is $A= A_k$: $A_0 = M$ is a field. Assume as induction hypothesis that $A_j$
is a division algebra  for some $j \in \{ 0, \ldots, k-1 \}$.
By the proof of Theorem \ref{thm:Petit (29)}, $\tau_j^{q_j}$ is
the inner automorphism $I_{c_j} ( z)= c_j z c_j^{-1}$ on $A_j$. Therefore
$$A_{j+1} \cong
A_j[t_j;\tau_j]/A_j[t_j;\tau_j](t_j^{q_j} - c_j)$$ is a division
algebra since $t_j^{q_j} - c_j \in A_j[t_j;\tau_j]$ is irreducible by
\cite[1.3.16]{J96}, because by assumption
$$b \tau_j(b) \cdots \tau_j^{q_j -1}(b) \neq c_j,$$ for all $b \in A_j$.
Thus $A_i$ is a division algebra for all $i \in \{ 0, \ldots, k \}$ by induction.
\end{proof}

The next result follows from Theorem \ref{thm:Petit (29) Generalised}. It generalizes \cite[(9)]{P68} and characterizes
 solvable crossed product algebras via generalized cyclic algebras:

\begin{corollary} \label{thm:Petit (29) Galois version}
Let $A$ be a crossed product algebra of degree $n$ over $F$ with
maximal subfield $M$ such that $M/F$ is a Galois field extension.
Then $G = {\rm Gal}(M/F)$ is solvable if there exists a chain of  subalgebras
$$M = A_0 \subset A_1 \subset \ldots \subset A_k = A$$
of $A$ which all have maximal subfield $M$, and are generalized cyclic algebras
$$A_{i+1} \cong A_i[t_i;\tau_i]/A_i[t_i;\tau_i](t_i^{q_i} - c_i),$$
over their centers  for all $i \in \{ 0, \ldots, k-1 \}$, where
  $q_i$ is a prime,
 $\tau_i$ is an $F$-automorphism of $A_i$ of inner order $q_i$ which restricts to an
automorphism $\sigma_{i+1} \in G$,
and $c_i \in {\rm Fix}(\tau_i)^\times$.
\end{corollary}

 \begin{remark}
Let $M/F$ be a finite Galois field extension with non-trivial solvable Galois group $G$ and $A$ a solvable crossed
product algebra over $F$ with maximal subfield $M$.

 A close inspection of Albert's proof \cite[p.~182-187]{albert1939structure} shows that he constructs
 the same chain of algebras
$$A_{i+1} \cong A_i[t_i;\tau_i]/A_i[t_i;\tau_i](t_i^{q_i} -c_i)$$
inside a solvable crossed product $A$ as we obtain in Theorem \ref{thm:main1}, but they are
not explicitly identified as generalized cyclic algebras.
 We also obtain a converse of Albert's statement (Corollary \ref{thm:Petit (29) Galois version}).
 \end{remark}

 Theorem \ref{thm:main1} also tells us something about the existence of $n$-central elements in
 a solvable crossed product algebra $A$, as $t_{k-1}$ is a
 $q_{k-1}$-central element in $A$. Recall that
 for a central simple algebra $A$ over $F$ whose degree is a multiple of $n$,
 $u\in A\setminus F$ is called an \emph{$n$-central element} if $u^n\in F^\times$ and
 $u^m\not\in F$ for all $1\leq m<n$. The
$n$-central elements play an important role in the structure of central simple algebras.

\begin{corollary} \label{cor:Petit (29) Galois version}
 Let $A$ be a solvable $G$-crossed product division algebra over $F$. Then
$$A \cong D[t;\tau]/D[t;\tau](t^{q} - c)=(D,\tau, c)$$ is a
generalized cyclic algebra, where $D$ is either a central simple
algebra over its center and $\tau$ a suitable automorphism of $D$ of finite inner order which is a prime $q$,
 or $D$ is a cyclic Galois field extension of $F$  of prime degree $q$ with Galois group $G=<\tau>$.
\\  $A$ contains a $q$-central element.
\end{corollary}

\begin{proof}
The first assertion follows directly from Theorem \ref{thm:main1}.
 In particular, then  $t\in A$ is a non-central element such that $t^{q}\in  F^\times$
  and $t^{m}\not\in F$ for all $1\leq m<q$.
\end{proof}

\section{Some simple consequences for admissible groups} \label{sec:app}

A finite group $G$ is called \emph{admissible} over a field $F$, if
there exists a  $G$-crossed product division algebra over $F$
\cite{schacher1968subfields}.

Suppose $G$ is a finite solvable group, so that we have a chain of normal subgroups
$$\{ 1\} = G_0 \leq \ldots \leq G_k =
G,$$
 where $G_{j} \triangleleft G_{j+1}$ and $G_{j+1} / G_{j}$ is cyclic of prime order $q_{j}$ for all
$j \in \{ 0, \ldots, k-1 \}$ as in (\ref{eqn:Subnormal series}) and (\ref{eqn:proof petit (29) 1}).

Suppose $G$ is admissible over  $F$. Then  Theorem \ref{thm:main1} shows that the subgroups $G_i$ of $G$ appearing in the chain of
normal subgroups of $G$ are admissible over suitable intermediate fields of $M/F$:

\begin{theorem} \label{thm:subgroups of G are admissible}
 Suppose $G$ is admissible over a field $F$. Then each $G_i$ in the above chain
is admissible over the intermediate field $Z_i={\rm Fix}(G_i)$ of
$M/F$
 and
 $$[Z_i:F] =\prod_{j=i}^{k-1} q_j,$$
  $i \in \{ 1, \ldots, k \}$.
  In particular, $G_{k-1}$ is admissible over   $Z_{k-1}={\rm Fix}(G_{k-1})$
  which has prime degree $q_{k-1}$ over $F$.
\end{theorem}

\begin{proof}
As $G$ is $F$-admissible there exists a $G$-crossed product division
algebra $A$ over $F$ and a chain of generalized cyclic division algebras
$$M = A_0\subset \ldots \subset A_k = A$$
over $F$, such that
\begin{equation}
A_{i+1} \cong A_i[t_i;\tau_i]/A_i[t_i;\tau_i](t_i^{q_i} -c_i)
\end{equation}
for all $i \in \{ 0, \ldots, k-1 \}$, where $\tau_i$ is an
automorphism of $A_i$ of inner order $q_i$ which restricts to an
automorphism $\sigma_{i+1} \in G$ and $c_i \in \mathrm{Fix}(\tau_i)$ is
invertible (Theorem \ref{thm:main1}).
 $A_i$ is a $G_i$-crossed product division algebra over $Z_i$ with maximal
subfield $M$ and $M/Z_i$ is a Galois field extension with
$\mathrm{Gal}(M/Z_i) = G_i$, i.e. $G_i$ is $Z_i$-admissible.
\end{proof}

\begin{example}
Let $G = \bf  S_4$, then $G$ is $\mathbb{Q}$-admissible \cite[Theorem
7.1]{schacher1968subfields}, so there exists a
 central simple division algebra $D$ over $\mathbb{Q}$ with
maximal subfield $M$, such that $M/F$ is a  Galois field
extension and ${\rm Gal}(M/F) = G$ is a finite solvable group. Let
 $$\{ {\rm id} \} \lhd <(12)(34)> \lhd \bf  K \lhd \bf  A_4 \lhd S_4$$
be its subnormal series, where $\bf K$ is the Klein four-group and $\bf A_4$ is the
alternating group, and
$${\bf S_4/A_4} \cong
\mathbb{Z}/2\mathbb{Z}, {\bf  A_4/K }\cong \mathbb{Z}/3\mathbb{Z}, \
{\bf K}/<(12)(34)> \cong \mathbb{Z}/2\mathbb{Z}, \ <(12)(34)>/\{ {\rm id}\} \cong \mathbb{Z}/2\mathbb{Z}.$$
By Corollary \ref{thm:Petit (29)
Galois version}, there exists a corresponding chain of division algebras
 $$M = A_0 \subset A_1 \subset A_2 \subset A_3 \subset A_4 =
D$$ over $\mathbb{Q}$, such that
$$A_{i+1} \cong A_i[t_i;\tau_i]/A_i[t_i;\tau_i](t_i^{q_i} - c_i)$$
is a generalized
cyclic division algebra over its center for all $i \in \{ 0, 1, 2, 3
\}$, where $\tau_i$ is an automorphism of $A_i$, whose restriction to
$M$ is $\sigma_{i+1} \in G$, $c_i \in {\rm Fix}(\tau_i)$, and $\tau_i$
has inner order $2, 2, 3, 2$ for $i = 0, 1, 2, 3$ respectively.
Moreover, we have $q_0 = q_1 = q_3 = 2$ and $q_2 = 3$, and $A_i$ has
degree $\prod_{l=0}^{i-1} q_l$ over its center $Z_i$ for all $i \in
\{ 1, 2, 3, 4 \}$ by Theorem \ref{thm:main1}. In addition, by Theorem
\ref{thm:subgroups of G are admissible} we conclude:
\begin{itemize}
\item[(i)] $\bf A_4$ is admissible over the quadratic field extension $Z_3={\rm Fix}({\bf A_4})\subset M$ of $\mathbb{Q}$.
\item[(ii)] $\bf K$ is admissible over the field extension $Z_2={\rm Fix}({\bf K})\subset M$  of $\mathbb{Q}$ of degree 6.
\item[(iii)] $<(12)(34)>$ is admissible over  the  field extension $Z_1={\rm Fix}(<(12)(34)>)\subset M$  of
$\mathbb{Q}$ of degree 12.
\end{itemize}
\end{example}

 Schacher proved that for every finite group $G$, there
exists an algebraic number field $F$ such that $G$ is admissible over
$F$ \cite[Theorem 9.1]{schacher1968subfields}. Combining this with
Theorem \ref{thm:main1} we obtain:

\begin{corollary}
Let $G$ be a finite solvable group. Then there exists an algebraic
number field $F$ and a  $G$-crossed product
division algebra $A$ over $F$. Furthermore, there exists a chain of
 division algebras
 $$M = A_0 \subset A_1 \subset \ldots \subset A_k = A$$
 over $F$, such that
 $$A_{i+1} \cong A_i[t_i;\tau_i]/A_i[t_i;\tau_i](t_i^{q_i} - c_i)$$
is a generalized cyclic algebra over its center $Z_i$ for all $i \in
\{ 0, \ldots, k-1 \}$,  satisfying the properties listed in
Theorems \ref{thm:Petit (29)}  and \ref{thm:main1}.
\\ In particular, each $G_i$ is admissible over $Z_i$.
\end{corollary}

\begin{proof}
Such a field $F$ and division algebra $D$ exist by \cite[Theorem
9.1]{schacher1968subfields}. The assertion follows by Corollary
\ref{thm:Petit (29) Galois version}.
\end{proof}

In \cite[Theorem 1]{sonn1983admissibility}, Sonn proved that a finite
solvable group is admissible over $\mathbb{Q}$ if and only if all its
Sylow subgroups are metacyclic, i.e. if every Sylow subgroup $H$ of
$G$ has a cyclic normal subgroup $N$, such that $H/N$ is also cyclic.
Combining this with Theorem \ref{thm:main1} we
conclude:

\begin{corollary} \label{cor:Q-admissible solvable petit (29) construction}
Let $G$ be a finite solvable group such that all its Sylow subgroups
are metacyclic. Then there exists a $G$-crossed
product division algebra $A$ over $\mathbb{Q}$, and a chain of division algebras
$$M = A_0 \subset A_1 \subset\ldots\subset A_k = A$$
over $\mathbb{Q}$, such that $$A_{i+1} \cong
A_i[t_i;\tau_i]/A_i[t_i;\tau_i](t_i^{q_i} - c_i)$$ is a generalized
cyclic algebra over its center $Z_i$ for all $i \in \{ 0, \ldots, k-1 \}$
satisfying the properties listed in Theorems \ref{thm:Petit (29)} and
\ref{thm:main1}.
\\ In particular, each $G_i$ is admissible over the field extension $Z_i$  of $\mathbb{Q}$.
\end{corollary}

\section{How to construct crossed product division algebras  containing a given abelian Galois field extension as a maximal
subfield} \label{sec:last}

For $S$ a unital ring and $\tau$ an injective endomorphism
of $S$,
$$f(t) =  t^q - c \in S[t;\tau]$$
 is an invariant twisted polynomial, if
 $$\tau^q(z) c = c \tau^i(z)\text{ and }\tau(c)
= c$$
 for all $z \in S$, $0\leq i<q$.

Let $M/F$ be a Galois field extension of degree $n$ with abelian Galois group $G =
\mathrm{Gal}(M/F)$. We now show how to canonically
construct crossed product division algebras of degree $n$ over $F$ containing $M$ as a
subfield. This generalizes a result by Albert  in which $n = 4$ and $G\cong\mathbb{Z}_2\times\mathbb{Z}_2 $
\cite[p. 186]{albert1939structure}, cf. also \cite[Theorem 2.9.55]{J96}:
For $n=4$ every central division algebra containing a quartic abelian extension $M$ with Galois group
$\mathbb{Z}_2\times\mathbb{Z}_2 $
 can be obtained this way \cite[p. 186]{albert1939structure}, that means as a generalized cyclic algebra $(D,\tau,c)$
 with $D$ a quaternion algebra over its center.

 Another
 way to construct such a crossed product algebra is via generic algebras, using a process going back to Amitsur
 and Saltman \cite{AmSalt}, described also in \cite[4.6]{J96}.

As $G$ is a finite abelian group, we have a chain of  subgroups
$$\{ 1 \} = G_0 \leq \ldots \leq G_k = G,$$
such that
$G_{j} \triangleleft G_{j+1}$ and $G_{j+1} / G_{j}$ is cyclic
of prime order $q_{j} > 1$ for all $j \in \{ 0, \ldots, k-1 \}$.
We use this chain to construct the algebras we want:

$G_1 = < \sigma_1>$ is cyclic of prime order $q_0 > 1$ for some $\sigma_1 \in G$. Let $\tau_0 = \sigma_1$.
Choose any $c_0 \in F^{\times}$ that satisfies
$$z \tau_0(z) \cdots \tau_0^{q_0-1}(z) \neq c_0$$
for all $z \in M$ and define
$$f(t_0)=t_0^{q_0}-c_0\in M[t_0;\tau_0].$$
Since  $\tau_0$ has order $q_0$, we have $\tau_0^{q_0}(z) c_0 = z c_0 = c_0 z$
for all $z \in M$, so that $f(t_0)$ is an invariant twisted polynomial and
hence
$$A_1 = M[t_0;\tau_0] /M[t_0;\tau_0](t_0^{q_0}-c_0)$$
is an associative  algebra which is cyclic of degree $q_0$
over ${\rm Fix}(\tau_0)$.
Moreover, $f(t_0)$ is irreducible by \cite[2.6.20 (i)]{J96}  and therefore $A_1$  is a division algebra.

Now  $G_2/G_1$ is cyclic of prime order $q_1$, say
$$G_2/G_1 = \{\sigma_2^i G_1 \ \vert \ i \in \mathbb{Z} \}$$
for some $\sigma_2 \in
G_2$ where $\sigma_2^{q_1} \in G_1$. As $\sigma_2^{q_1} \in G_1$ we
have $\sigma_2^{q_1} = \sigma_1^{\mu}$ for some $\mu \in \{ 0,
\ldots, q_0-1 \}$.
Define $c_1 = l_1 t_0^{\mu}$ for some $l_1 \in F^{\times}$
and define the map $$\tau_1
: A_1\rightarrow A_1, \ \sum_{i=0}^{q_0-1} m_i t_0^i \mapsto
\sum_{i=0}^{q_0-1} \sigma_2(m_i) t_0^i,$$ which is an automorphism of
$A_1$ by a straightforward calculation.

 Denote the multiplication in $A_1$ by $\circ$. Then
$$\tau_1(c_1) =\sigma_2(l_1) t_0^{\mu} = l_1 t_0^{\mu} = c_1.$$
We have
\begin{align*}
\tau_1^{q_1} \Big( \sum_{i=0}^{q_0-1} m_i t_0^i \Big) \circ c_1
&= \sum_{i=0}^{q_0-1} \sigma_2^{q_1}(m_i) t_0^i \circ l_1
t_0^{\mu} = \sum_{i=0}^{q_0-1} l_1 \sigma_1^{\mu}(m_i) t_0^i
\circ t_0^{\mu} \\
\end{align*}
and
\begin{align*}
c_1 \circ \sum_{i=0}^{q_0-1} m_i t_0^i &= l_1 t_0^{\mu}
\circ \sum_{i=0}^{q_0-1} m_i t_0^i = \sum_{i=0}^{q_0-1} l_1
\sigma_1^{\mu}(m_i) t_0^{\mu} \circ t_0^i \\
\end{align*}
for all $m_i \in M$. Hence $\tau_1^{q_1}(z) \circ c_1 = c_1
\circ z$ for all $z \in A_1$ and $\tau_1(c_1) = c_1$, thus
$$f(t_1)=t_1^{q_1}-c_1\in A_1[t_1;\tau_1]$$
is an invariant  twisted polynomial and
$$A_2 = A_1[t_1;\tau_1]/A_1[t_1;\tau_1](t_1^{q_1}-c_1)$$
is a finite-dimensional associative algebra over
${\rm Fix}(\tau_1)\cap C(A_1)={\rm Fix}(\tau_1)\cap {\rm Fix}(\tau_0)\supset F$ \cite{Pu16}.

Again, $G_3/G_2$ is cyclic of prime order $q_2$, say
$$G_3/G_2 = \{\sigma_3^i G_2 \ \vert \ i \in \mathbb{Z} \}$$
for some $\sigma_3 \in
G$ with $\sigma_3^{q_2} \in G_2$. Write
$$\sigma_3^{q_2} =
\sigma_2^{\lambda_1} \sigma_1^{\lambda_0}$$
for some $\lambda_1 \in \{
0, \ldots, q_1-1 \}$ and $\lambda_0 \in \{ 0, \ldots, q_0-1 \}$.
The map $$H_{\sigma_3}: A_1 \rightarrow A_1, \
\sum_{i=0}^{q_0-1} m_i t_0^i \mapsto \sum_{i=0}^{q_0-1} \sigma_3(m_i)t_0^i,$$
is an automorphism of $A_1$ by
a straightforward calculation. Define
$$\tau_2:
A_2 \rightarrow A_2, \ \sum_{i=0}^{q_1-1} x_i t_1^i \mapsto
\sum_{i=0}^{q_1-1} H_{\sigma_3}(x_i) t_1^i \ (x_i \in A_1).$$
Then a straightforward calculation using that $H_{\sigma_3}$ commutes with $\tau_1$ and $H_{\sigma_3}(c_1) =
c_1$ shows that $\tau_2$ is an automorphism of $A_2$. Define
$$c_2 = l_2 t_0^{\lambda_0} t_1^{\lambda_1}$$
for some $l_2 \in F^{\times}$.
Denote the
multiplication in $A_i$ by $\circ_{A_i}$ and let $x_i =
\sum_{j=0}^{q_0-1} y_{ij} t_0^j \in A_1$, $y_{ij} \in M$, $i \in \{
0, \ldots, q_1-1 \}$. Then
$$\tau_2(c_2) = \tau_2(l_2 t_0^{\lambda_0}
t_1^{\lambda_1}) = H_{\sigma_3}(l_2 t_0^{\lambda_0})
t_1^{\lambda_1} = l_2 t_0^{\lambda_0} t_1^{\lambda_1} = c_2.$$
Furthermore we have
\begin{align*}
\tau_2^{q_2} \Big( \sum_{i=0}^{q_1-1} x_i t_1^i \Big) \circ_{A_2} c_2
&= \sum_{i=0}^{q_1-1} H_{\sigma_3}^{q_2}(x_i) t_1^i \circ_{A_2} l_2
t_0^{\lambda_0} t_1^{\lambda_1} \\ &= \sum_{i=0}^{q_1-1}
\sum_{j=0}^{q_0-1} \sigma_3^{q_2}(y_{ij}) t_0^j t_1^i \circ_{A_2} l_2
t_0^{\lambda_0} t_1^{\lambda_1} \\ &= \sum_{i=0}^{q_1-1}
\sum_{j=0}^{q_0-1} \sigma_2^{\lambda_1}(\sigma_1^{\lambda_0}(y_{ij}))
t_0^j t_1^i \circ_{A_2} l_2 t_0^{\lambda_0} t_1^{\lambda_1} \\ &=
\sum_{i=0}^{q_1-1} \Big( \sum_{j=0}^{q_0-1}
\sigma_2^{\lambda_1}(\sigma_1^{\lambda_0}(y_{ij})) t_0^j \circ_{A_1}
\tau_1^i(l_2 t_0^{\lambda_0}) \Big) t_1^i \circ_{A_2} t_1^{\lambda_1}
\\ &= \sum_{i=0}^{q_1-1} \Big( \sum_{j=0}^{q_0-1} l_2
\sigma_2^{\lambda_1}(\sigma_1^{\lambda_0}(y_{ij})) t_0^j \circ_{A_1}
t_0^{\lambda_0} \Big) t_1^i \circ_{A_2} t_1^{\lambda_1},
\end{align*}
and
\begin{align*}
c_2 \circ_{A_2} \sum_{i=0}^{q_1-1} x_i t_1^i &= l_2 t_0^{\lambda_0}
t_1^{\lambda_1} \circ_{A_2} \sum_{i=0}^{q_1-1} x_i t_1^i =
\sum_{i=0}^{q_1-1} \big( l_2 t_0^{\lambda_0} \circ_{A_1}
\tau_1^{\lambda_1}(x_i) \big) t_1^{\lambda_1} \circ_{A_2} t_1^i \\ &=
\sum_{i=0}^{q_1-1} \sum_{j=0}^{q_0-1} \big( l_2 t_0^{\lambda_0}
\circ_{A_1} \sigma_2^{\lambda_1}(y_{ij}) t_0^j \big) t_1^{\lambda_1}
\circ_{A_2} t_1^i \\ &= \sum_{i=0}^{q_1-1} \Big( \sum_{j=0}^{q_0-1}
l_2 \sigma_1^{\lambda_0}(\sigma_2^{\lambda_1}(y_{ij}))
t_0^{\lambda_0} \circ_{A_1} t_0^j \Big) t_1^{\lambda_1} \circ_{A_2}
t_1^i.
\end{align*}
Hence $\tau_2^{q_2}(z) \circ_{A_2} c_2 = c_2 \circ_{A_2} z$ for all $z \in A_2$ and $\tau_2(c_2) = c_2$, therefore
$$f(t_2)=t_2^{q_2}-c_2\in A_2[t_2;\tau_2]$$
is an invariant  twisted polynomial and thus
$$A_3=A_2[t_2;\tau_2]/A_2[t_2;\tau_2](t_2^{q_2}-c_2)$$
is a finite-dimensional associative  algebra over
  $${\rm Fix}(\tau_2) \cap C(A_2) = {\rm Fix}(\tau_0)\cap {\rm Fix}(\tau_1)\cap {\rm Fix}(\tau_2)\supset F$$
   \cite{Pu16}. Continuing in this manner we obtain a chain
$M = A_0 \subset \ldots \subset A_k$
of finite-dimensional associative algebras
$$A_{i+1} = A_i[t_i;\tau_i]/A_i[t_i;\tau_i](t_i^{q_i}-c_i)$$
over
$${\rm Fix}(\tau_i)\cap C(A_i) = {\rm Fix}(\tau_0)\cap {\rm Fix}(\tau_1)\cap \dots\cap{\rm Fix}(\tau_i)\supset F,$$
for all $i \in \{ 0, \ldots, k-1 \}$, where $\tau_0 = \sigma_1$ and
$\tau_i$ restricts to $\sigma_{i+1}$ on $M$ for all $i \in \{ 0, \ldots,k-1 \}$. Moreover,
$$[A_i:M] = [A_i: A_{i-1}] \cdots [A_1:M] =\prod_{l=0}^{i-1} q_l$$
hence
$$[A_k:F] = \big( \prod_{l=0}^{k-1} q_l\big) n = n^2,$$
and $A_k$ contains $M$ as a subfield.

\begin{lemma} \label{lem:tau_i has inner order q_i}
For all $i \in \{ 0, \ldots, k-1 \}$, $\tau_i:A_i \rightarrow A_i$ has inner order $q_i$.
\end{lemma}

\begin{proof}

The automorphism $\tau_0 = \sigma_1: M \rightarrow M$ has inner order $q_0$.

Fix $i \in \{ 1, \ldots, k-1 \}$.
 $A_i$ is finite-dimensional over $F$, so it is also finite-dimensional over its center $C(A_i) \supset F$.
Recall that $\tau_i^{q_i}(z)  c_i = c_i  z$ for all $z \in A_i$,
in particular $\tau_i^{q_i}\vert_{C(A_i)} = id $. As $q_i$ is prime this
means either $\tau_i \vert_{C(A_i)} = id $ or $\tau_i \vert_{C(A_i)}$ has order $q_i>1$.

Assume  that $\tau_i \vert_{C(A_i)} = id $, then $\tau_i$ is
an inner automorphism of $A_i$ by the Theorem of   Skolem-Noether, say $\tau_i(z) = uzu^{-1}$ for some
invertible $u \in A_i$, for all $z \in A_i$.
In particular $\tau_i(m)= \sigma_{i+1}(m) = umu^{-1}$ for all $m \in M$. Write
$$u =
\sum_{j=0}^{q_{i-1}-1} u_j t_{i-1}^j$$
for some $u_j \in A_{i-1}$, thus
\begin{align*}
\sigma_{i+1}(m)u &= \sigma_{i+1}(m) \sum_{j=0}^{q_{i-1}-1} u_j t_{i-1}^j =
\sum_{j=0}^{q_{i-1}-1} u_j t_{i-1}^j m \\ &= \sum_{j=0}^{q_{i-1}-1}
u_j \tau_{i-1}^j(m) t_{i-1}^j = \sum_{j=0}^{q_{i-1}-1} u_j
\sigma_{i}^j(m) t_{i-1}^j.
\end{align*}
for all $m \in M$. Choose $\eta_{i}$ with $u_{\eta_{i}}
\neq 0$ then
\begin{equation} \label{eqn:tau_i is not inner 1}
\sigma_{i+1}(m) u_{\eta_{i}} = u_{\eta_{i}}
\sigma_{i}^{\eta_{i}}(m),
\end{equation}
for all $m \in M$.

If $i = 1$ we are done. If $i \geq 2$ then
we can also write $u_{\eta_{i}} =
\sum_{l=0}^{q_{i-2}-1} w_l t_{i-2}^l$ for some $w_l \in A_{i-2}$,
therefore \eqref{eqn:tau_i is not inner 1} yields
\begin{align*}
\sigma_{i+1}(m) \sum_{l=0}^{q_{i-2}-1} w_l t_{i-2}^l &=
\sum_{l=0}^{q_{i-2}-1} w_l t_{i-2}^l \sigma_{i}^{\eta_{i}}(m)
= \sum_{l=0}^{q_{i-2}-1} w_l
\tau_{i-2}^l(\sigma_{i}^{\eta_{i}}(m)) t_{i-2}^l \\ &=
\sum_{l=0}^{q_{i-2}-1} w_l
\sigma_{i-1}^l(\sigma_{i}^{\eta_{i}}(m)) t_{i-2}^l,
\end{align*}
for all $m \in M$. Choose $\eta_{i-1}$ with $w_{\eta_{i-1}}
\neq 0$,  then $$\sigma_{i+1}(m) w_{\eta_{i-1}} = w_{\eta_{i-1}}
\sigma_{i-1}^{\eta_{i-1}}(\sigma_{i}^{\eta_{i}}(m)),$$ for
all $m \in M$.

Continuing in this manner we see that there exists $s \in M^{\times}$ such
that $$\sigma_{i+1}(m) s = s
\sigma_1^{\eta_1}(\sigma_2^{\eta_2}(\cdots
(\sigma_{i}^{\eta_{i}}(m) \cdots ),$$ for all $m \in M$, hence
$$\sigma_{i+1}(m) = \sigma_1^{\eta_1}(\sigma_2^{\eta_2}(\cdots
(\sigma_{i}^{\eta_{i}}(m) \cdots ),$$ for all $m \in M$ where
$\eta_j \in \{ 0, \ldots, q_{j-1}-1 \}$ for all $j \in \{ 1, \ldots,
i \}$. But $\sigma_{i+1} \notin G_i$ and thus
$$\sigma_{i+1} \neq \sigma_1^{\eta_1}\circ\sigma_2^{\eta_2}\circ \cdots \circ\sigma_{i}^{\eta_{i}},$$
a contradiction.

It follows that  $\tau_i \vert_{C(A_i)}$ has order  $q_i>1$. By the
Skolem-Noether Theorem the kernel of the restriction map
$\mathrm{Aut}(A_i) \rightarrow \mathrm{Aut}(C(A_i))$ is
the group of inner automorphisms of $A_i$, and so $\tau_i$ has inner order $q_i$.
\end{proof}

Let us furthermore assume that each $c_i$ above, $i \in \{ 0, \ldots, k-1 \}$,  is successively chosen such that
\begin{equation}\label{equ:last}
z \tau_i(z) \cdots \tau_i^{q_i-1}(z)\neq c_i
\end{equation}
for all $z \in A_i$, then using that $\tau_i$ has inner order $q_i$,
$$f(t_i)=t_i^{q_i}-c_i\in A_i[t_i;\tau_i]$$
is an irreducible twisted polynomial by Lemma \ref{lem:tau_i has inner order q_i} and thus $A_{i+1}$  is  a division algebra \cite[1.3.16]{J96}.

\begin{proposition}
$C(A_k)=F$.
\end{proposition}

\begin{proof}
  $F \subset C(A_k)$ by construction.
Let now
$$z = z_0 + z_1 t_{k-1} + \ldots + z_{q_{k-1}-1} t_{k-1}^{q_{k-1}-1} \in C(A_k)$$
where
$z_i \in A_{k-1}$. Then  $z$ commutes with all $l \in A_{k-1}$, hence $l z_i = z_i \tau_{k-1}^i(l)$ for all $i \in \{ 0,
\ldots, q_{k-1}-1 \}$. This implies $z_0 \in C(A_{k-1})$
and $z_i = 0$ for all $i \in \{ 1, \ldots, q_{k-1}-1 \}$, otherwise
$z_i$ is invertible and $\tau_{k-1}^i$ is inner, a contradiction by
Lemma \ref{lem:tau_i has inner order q_i}. Thus $z = z_0 \in
C(A_{k-1})$. A similar argument shows $z \in
C(A_{k-1})$ and continuing in this manner we conclude $z
\in  M=C(A_0)$.

Suppose for contradiction $z \notin F$, then $\rho(z) \neq z$ for
some $\rho \in G$. Since the $\sigma_{i+1}$ were chosen so that
 they generate the cyclic factor groups $ G_{i+1}/G_i,$ we can write
 $\rho =
\sigma_1^{i_0} \circ\sigma_2^{i_1} \circ\cdots \circ \sigma_{k}^{i_{k-1}}$ for some
$i_s \in \{ 0, \ldots, q_s-1 \}$. We have
\begin{align*}
t_0^{i_0} t_1^{i_1} \cdots t_{k-1}^{i_{k-1}} z &=
\sigma_1^{i_0}(\sigma_2^{i_1}(\cdots(\sigma_{k}^{i_{k-1}}(z) \cdots
) t_0^{i_0} t_1^{i_1} \cdots t_{k-1}^{i_{k-1}} \\ &= \rho(z)
t_0^{i_0} t_1^{i_1} \cdots t_{k-1}^{i_{k-1}} \neq z t_0^{i_0}
t_1^{i_1} \cdots t_{k-1}^{i_{k-1}},
\end{align*}
contradicting the assumption that $z \in C(A_k)$. Therefore $C(A_k) \subset F$.
\end{proof}

This yields a recipe for constructing a $G$-crossed product division algebra $A=A_k$ over $F$ with maximal
subfield $M$ provided it is possible to find suitable $c_i$'s satisfying (\ref{equ:last}).

 By Corollary \ref{thm:Petit (29) Galois version}, every abelian crossed product division algebra that is solvable
 can be obtained this way, starting with a suitable $M/F$.

%*******************************************************************************************%
%****************************************************************************************%

\end{document}